\documentclass[oneside,english]{amsart}
\usepackage[T1]{fontenc}
\usepackage[latin9]{inputenc}
\usepackage{geometry}
\usepackage{float}
\usepackage{amsthm}
\usepackage{amssymb}
\usepackage{graphicx}
\usepackage{esint}

\makeatletter

\newcommand{\lyxdot}{.}

\numberwithin{equation}{section}
\numberwithin{figure}{section}
  \theoremstyle{plain}
  \newtheorem*{thm*}{\protect\theoremname}
\theoremstyle{plain}
\newtheorem{thm}{\protect\theoremname}
  \theoremstyle{plain}
  \newtheorem{lem}[thm]{\protect\lemmaname}
  \theoremstyle{definition}
  \newtheorem{defn}[thm]{\protect\definitionname}
  \theoremstyle{plain}
  \newtheorem{cor}[thm]{\protect\corollaryname}
  \theoremstyle{definition}
  \newtheorem{problem}[thm]{\protect\problemname}
  \theoremstyle{plain}
  \newtheorem{prop}[thm]{\protect\propositionname}

\makeatother

\usepackage{babel}
  \providecommand{\corollaryname}{Corollary}
  \providecommand{\definitionname}{Definition}
  \providecommand{\lemmaname}{Lemma}
  \providecommand{\problemname}{Problem}
  \providecommand{\propositionname}{Proposition}
  \providecommand{\theoremname}{Theorem}
\providecommand{\theoremname}{Theorem}

\begin{document}

\title{On Polygons Admitting a Simson Line as Discrete Analogs of Parabolas}

\author{Emmanuel Tsukerman}

\date{\today}

\maketitle

\section{Introduction}

The Simson-Wallace Theorem (see, e.g., \cite{Riegel}) is a classical
result in plane geometry. It states that
\begin{thm*}
(Simson-Wallace Theorem%
\footnote{One remark concerning the theorem is that the Simson-Wallace Theorem
is most commonly known as ``Simson's Theorem'', even though ``Wallace
is known to have published the theorem in $1799$ while no evidence
exists to support Simson's having studied or discovered the lines
that now bear his name'' \cite{Riegel}. This is perhaps one of the
many examples of Stigler's law of eponymy.%
}). \label{(Simson-Wallace-Theorem)}Given a triangle $\triangle ABC$
and a point $P$ in the plane, the pedal points of $P$ (That is,
the feet of the perpendiculars dropped from $P$ to the sides of the
triangle) are collinear if and only if $P$ is on the circumcircle
of $\triangle ABC$.
\end{thm*}
\noindent \begin{flushleft}
Such a line is called a Simson line of $P$ with respect to $\triangle ABC$. 
\par\end{flushleft}

A natural question is whether an $n$-gon with $n\geq4$ can admit
a Simson line. In \cite{Johnson-book} and \cite{RadkoTsukerman},
it is shown that every quadrilateral possesses a unique Simson Line,
called ``the Simson Line of a complete%
\footnote{A complete quadrilateral is the configuration formed by $4$ lines
in general position and their $6$ intersections. When it comes to
pedals, we are only concerned with the sides making up the polygon.
Since we extend these, the pedal of a quadrilateral is equivalent
to that of its complete counterpart. For this reason, we will refer
to a polygon simply by the number of sides it has.%
} quadrilateral''. We call a polygon which admits a Simson line a
\textit{Simson polygon}. In this paper, we show that there is a strong
connection between Simson polygons and the seemingly unrelated parabola.

We begin by proving a few general facts about Simson polygons. We
use an inductive argument to show that no convex $n$-gon, $n\geq5$,
admits a Simson Line. We then determine a property which characterizes
Simson $n$-gons and show that one can be constructed for every $n\geq3$.
We proceed to show that a parabola can be viewed as a limit of special
Simson polygons, called \textit{equidistant Simson polygons}, and
that these polygons provide the best piecewise linear continuous approximation
to the parabola. Finally, we show that equidistant Simson polygons
can be viewed as discrete analogs of parabolas and that they satisfy
a number of results analogous to the pedal property, optical property,
properties of Archimedes triangles and Lambert's Theorem of parabolas.
The corresponding results for parabolas are easily obtained by applying
a limit process to the equidistant Simson polygons.

\section{General Properties of Simson Polygons}

We begin with an easy Lemma. Throughout, we will use the notation
that $(XYZ)$ is the circle through points $X,Y,Z$. 
\begin{lem}
\label{lem:angle circle and distances}Let $S$ be a point in the
interior of two rays $AB$ and $AC$. Suppose that $ABSC$ is cyclic,
and let $X$ be a point on ray $AB$ such that $|AX|<|AB|$. Let $Y=(AXS)\cap AC$.
Then $|AY|>|AC|$.\end{lem}
\begin{proof}
Since $|AX|<|AB|$, $\angle AXS>\angle ABS$. Since $ABSC$ and $AXSY$
are cyclic, $\angle ACS=\pi-\angle ABS$ and $\angle AYS=\pi-\angle AXS$.
Therefore $\angle AYS<\angle ABS$ so that $|AY|>|AC|$.

\begin{figure}[h]
\includegraphics[scale=0.5]{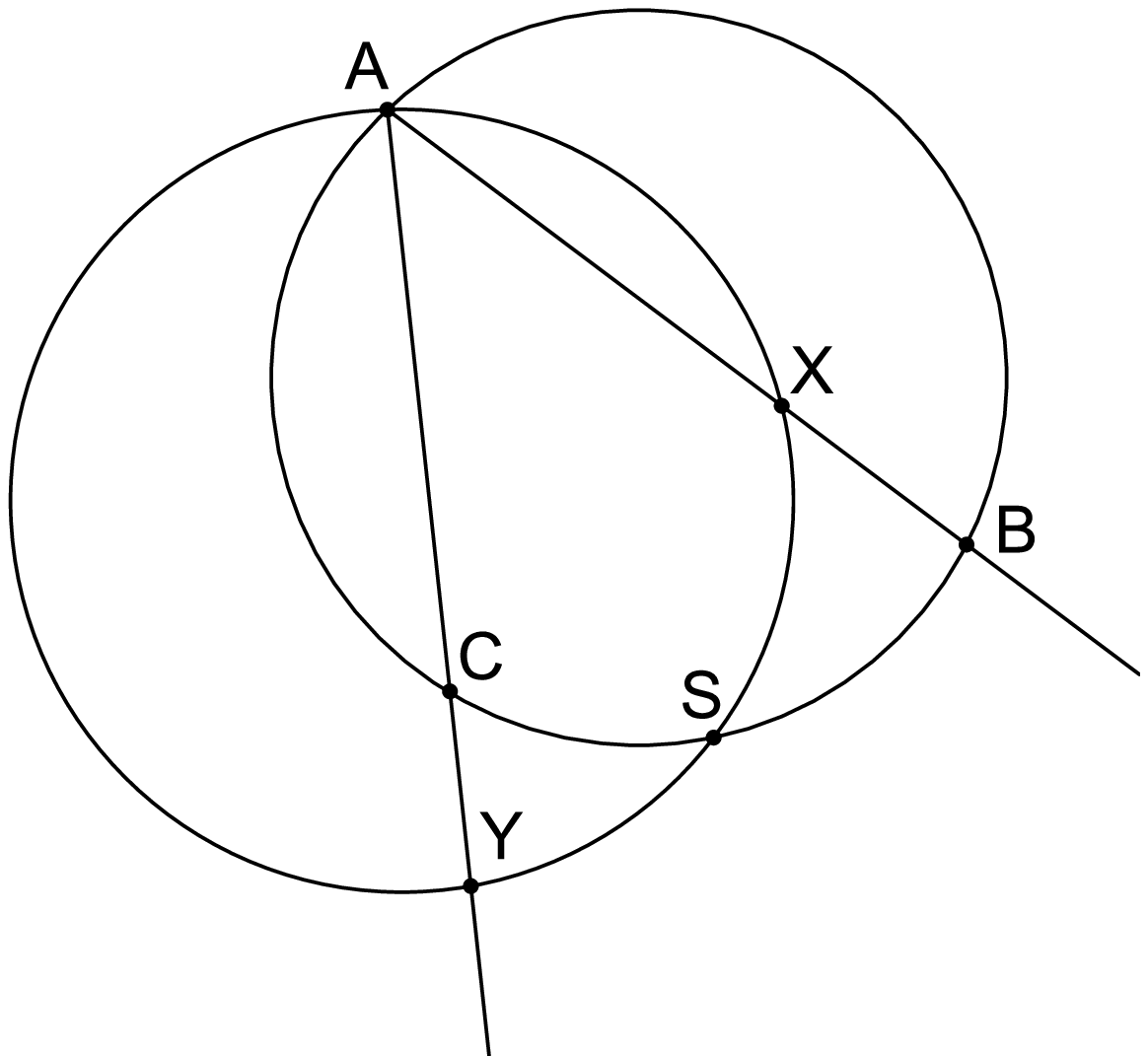}\includegraphics[scale=0.5]{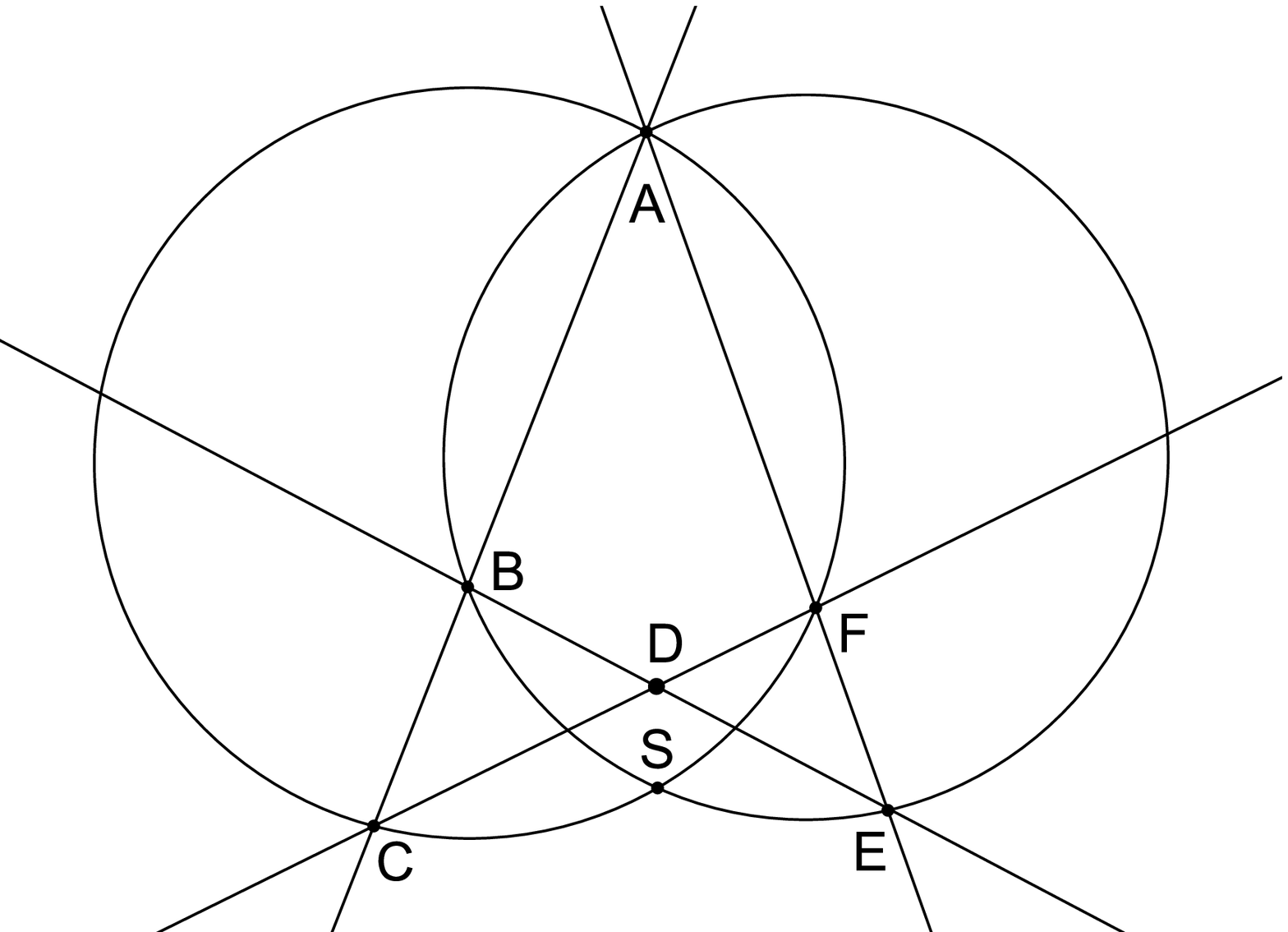}

\caption{\label{fig:Lemma 1 and 2}Lemma \ref{lem:angle circle and distances}
and Lemma \ref{cor:complete quad}}

\end{figure}

\end{proof}
As mentioned in the introduction, in the case of a quadrilateral there
is always a unique Simson point defined as a point from which the
projections into the sides are collinear. Let $A,B,C,D,E,F$ denote
the vertices of the complete quadrilateral, as in fig. \ref{fig:Lemma 1 and 2}.
It is shown in \cite{Johnson-book} that the Simson point is the unique
intersection of $(AFC)\cap(ABE)\cap(BCD)\cap(DEF)$, also known as
the \textit{Miquel point of a complete quadrilateral}. Using Lemma
\ref{lem:angle circle and distances}, we can conclude the following:
\begin{lem}
\label{cor:complete quad}Let $ABCDEF$ be a complete quadrilateral
where points in each of the triples $A,B,C$; $B,D,E$, etc. as in
fig. \ref{fig:Lemma 1 and 2} are collinear and angle $\angle CDE$
is obtuse. Denote the Miquel point of $ABCDEF$ by $S$. There exist
no two points $X$ and $Y$ on rays $AF$, $AB$ respectively with
$|AX|<|AF|$, $|AY|<|AB|$ such that $(AXY)$ passes through $S$.\end{lem}
\begin{proof}
The Miquel point $S$ lies on $(AFC)$ and $(ABE)$. By Lemma \ref{lem:angle circle and distances},
no such $X$ and $Y$ exist.
\end{proof}

We call a polygon for which no three vertices lie on a line nondegenerate.
In Lemma \ref{lem:no pair of parallel sides} and Theorems \ref{thm:convex pentagon does not admit simson}
and \ref{thm:convex ngon does not admit simson} we will assume that
the polygon is nondegenerate.
\begin{lem}
\label{lem:no pair of parallel sides}If $\Pi=V_{1}\cdot\cdot\cdot V_{n}$,
$n\geq5$ is a convex Simson polygon, then $\Pi$ has no pair of parallel
sides. \end{lem}
\begin{proof}
By the nondegeneracy assumption, it is clear that no two consecutive
sides can be parallel. So suppose that $V_{1}V_{2}\parallel V_{i}V_{i+1}$,
$i\notin\{1,2,n\}$. Then $S$ lies on the Simson line $L$ orthogonal
to $V_{1}V_{2}$ and $V_{i}V_{i+1}$. The projection of $S$ into
each other side $V_{j}V_{j+1}$ must also lie on $L$, so that either
$V_{j}V_{j+1}$ is parallel to $V_{1}V_{2}$ or it passes through
$S$. By the nondegeneracy assumption, no two consecutive sides can
pass through $S$. Therefore the sides of $\Pi$ must alternate between
being parallel to $V_{1}V_{2}$ and passing through $S$. It is easy
to see that no such polygon can be convex. 
\end{proof}
It is worth noting that both the convexity hypothesis and the restriction
to $n\geq5$ in the last result are necessary, for one can construct
a non-convex $n$-gon, $n\geq5$ having pairs of parallel sides and
the trapezoid (if not a parallelogram) is a convex Simson polygon
with $n=4$ having a pair of parallel sides. Using the above result,
we can prove:
\begin{thm}
\label{thm:convex pentagon does not admit simson}A convex pentagon
does not admit a Simson point.\end{thm}
\begin{proof}
Let $\Pi=ABCDE$ be a nondegenerate convex pentagon. Suppose that
$S$ is a point for which the pedal in $\Pi$ is a line. Then in particular
the pedal is a line for every 4 sides of the pentagon. Therefore if
$BC\cap DE=F$, then $S$ must be a Simson point for $ABFE$, so that
$S$ is the Miquel point of $ABFE$. This implies that 
\[
S=(GAB)\cap(GFE)\cap(HAE)\cap(HBF),
\]
where $BC\cap AE=G$ and $AB\cap DE=H$. By the same reasoning applied
to quadrilateral $CGED$, $S$ must be the Miquel point of $CGED$.
Therefore $S$ lies on $(FCD)$. Because $\Pi$ is convex, $|FC|<|FB|$
and $|FD|<|FE|$. We can now apply corollary \ref{cor:complete quad}
with $C$ and $D$ playing the role of points $X$ and $Y$ to conclude
that $S$ cannot lie on $(FCD)$ - a contradiction. 
\end{proof}
Consider a convex polygon $\Pi$ as the boundary of the intersection
of half planes $H_{1},H_{2}$,...,$H_{n}$. Then the polygon formed
from the boundary of $\underset{\underset{i=1}{i\neq k}}{\overset{n}{\cap}}H_{i}$
for $k\in\{1,2,...,n\}$ is also convex. 

We are now ready to prove the following result by induction:
\begin{thm}
\label{thm:convex ngon does not admit simson}A convex $n$-gon with
$n\geq5$ does not admit a Simson point.\end{thm}
\begin{proof}
The base case has been established. Assume the hypothesis for $n\geq5$,
and consider the case for an $(n+1)$-gon $\Pi$ with vertices $V_{1},...,V_{n+1}$.
Suppose that $\Pi$ admits a Simson point. Let $V_{n-1}V_{n}\cap V_{n+1}V_{1}=V'$.
This intersection exists by Lemma \ref{lem:no pair of parallel sides}.
Since $\Pi$ admits a Simson point, $\Pi'=V_{1}...V_{n-1}V'$ must
also admit one. By the preceding remark, $\Pi'$ is convex, and since
it has $n$ sides, the hypothesis is contradicted. Therefore $\Pi$
cannot admit a Simson line, completing the induction.
\end{proof}
Now that we have established that no convex $n$-gon (with $n\geq5$)
admits a Simson line, we will proceed to find a necessary and sufficient
condition for an $n$-gon $\Pi=V_{1}V_{2}...V_{n}$ to have a Simson
point. Let $W_{i}=V_{i-1}V_{i}\cap V_{i+1}V_{i+2}$ for each $i$,
with $V_{n+k}=V_{k}$. In case that $V_{i-1}V_{i}$ and $V_{i+1}V_{i+2}$
are parallel, view $W_{i}$ as a point at infinity and $(V_{i}W_{i}V_{i+1})$
as the line $V_{i}V_{i+1}$. For example, in a right-angled trapezoid
with $AB\perp BC$ and $AB\perp AD$, $S$ will necessarily lie on
the line $AB$ (in fact $S=AB\cap CD$).
\begin{thm}
An $n$-gon $\Pi=V_{1}\cdot\cdot\cdot V_{n}$ admits a Simson point
$S$ if and only if all circles $(V_{i}W_{i}V_{i+1})$ have a common
intersection.\end{thm}
\begin{proof}
Assume first that $S$ is a Simson point for $\Pi$. The projections
of $S$ into $V_{i-1}V_{i}$, $V_{i}V_{i+1}$ and $V_{i+1}V_{i+2}$
are collinear. By the Simson-Wallace Theorem (Theorem \ref{(Simson-Wallace-Theorem)}),
$S$ is on the circumcircle of $V_{i}W_{i}V_{i+1}$.

Conversely, let $S=\underset{i}{\cap}(V_{i}W_{i}V_{i+1})$. For each
$i$, this implies that the projections of $S$ into $V_{i-1}V_{i}$
and $V_{i+1}V_{i+2}$ are collinear. As $i$ ranges from $1$ to $n$
we see that all projections of $S$ into the sides are collinear.
\end{proof}
To construct an $n$-gon with a given Simson point $S$ and Simson
line $L$, let $X_{1}$, $X_{2}$,...,$X_{n}$ be $n$ points on $L$.
The $n$ lines the $i$th of which is perpendicular to $SX_{i}$ and
passing through $X_{i}$, $i=1,...,n$ are the sides of an $n$-gon
with Simson point $S$ and Simson line $L$. The $X_{i}$ are the
projections of $S$ into the sides of the $n$-gon and the vertices
are $V_{i}=SX_{i}\cap SX_{i+1}$, $i=1,...,n$.

\begin{figure}[H]
\includegraphics[scale=0.4]{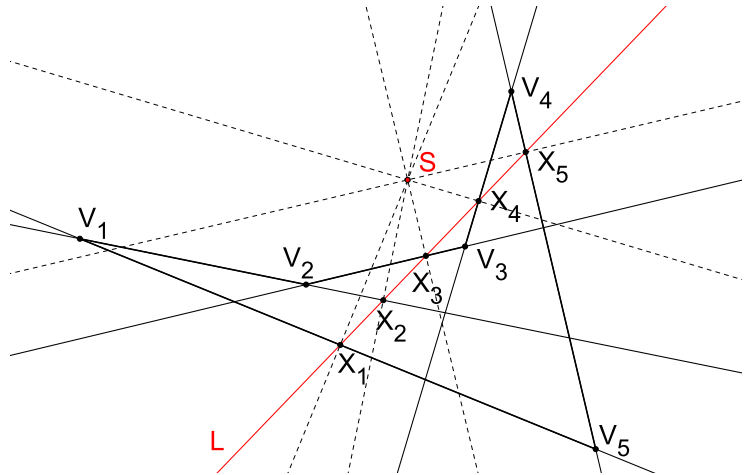}

\caption{A construction of a pentagon $V_{1}V_{2}V_{3}V_{4}V_{5}$ with Simson
point $S$ and Simson line $L$. The points $X_{1},...,X_{5}$ on
$L$ are the projections of $S$ into the sides of the pentagon. }
\end{figure}

\section{Simson Polygons and Parabolas}

In this section we will show that there is a strong connection between
Simson polygons and parabolas. In particular, we may view a special
type of Simson polygons, which we call equidistant Simson polygons,
as discrete analogs of the parabola. 
\begin{defn}
Let $\Pi=V_{1}\cdot\cdot\cdot V_{n}$ be a Simson polygon with Simson
point $S$ and projections $X_{1},...,X_{n}$ of $S$ into its sides.
In the special case that $|X_{i}X_{i+1}|=\Delta$ for each $i=1,...,n-1$,
we call such a polygon $\Pi$ an \textit{equidistant Simson polygon}. 
\end{defn}
The following result shows that all but one of the vertices of an
equidistant Simson polygon lie on a parabola. Moreover, the parabola
is independent of the position of $X_{1}$ (but depends on $\Delta)$. 
\begin{thm}
\label{thm:Parabola}Let $S$ be a point and $L$ a line not passing
through $S$. Suppose that $X_{1},...,X_{n}$ are points on $L$ such
that $|X_{i}X_{i+1}|=\Delta$ for all $i=1,...,n-1$ and let $\Pi=V_{1}\cdot\cdot\cdot V_{n}$
be the equidistant Simson polygon with Simson point $S$ and projections
$X_{1},...,X_{n}$ of $S$ into its sides. Then $V_{1},...,V_{n-1}$
lie on a parabola $C$. Moreover, $C$ is independent of the position
of $X_{1}$ on $L$.\end{thm}
\begin{proof}
Without loss of generality, let $S=(0,s)$, $L$ be the $x$-axis,
$X_{i}=(X+(i-1)\Delta,0)$ and $X_{i+1}=(X+i\Delta,0$). A calculation
shows that the perpendiculars at $X_{i}$ and $X_{i+1}$ to the segments
$SX_{i}$ and $SX_{i+1}$, respectively, intersect at the point $(2X+(2i-1)\Delta,\frac{(X+(i-1)\Delta)(X+i\Delta)}{s})$.
Therefore the coordinates of the intersection satisfy $y=\frac{x^{2}-\Delta^{2}}{4s}$
independently of $X$. It follows that $V_{1},...,V_{n-1}$ lie on
the parabola $y=\frac{x^{2}-\Delta^{2}}{4s}$ . 
\end{proof}
The fact that $C$ is independent of the position of $X_{1}$ on $L$
can be illustrated on figure \ref{fig:equidistant Simson on parabola}
by supposing that $X_{1},X_{2},...,X_{8}$ are being translated on
$L$ as a rigid body. Then the independence of $C$ from $X_{1}$
implies that $C$ remains fixed and $V_{1},...,V_{7}$ slide together
about $C$.

\begin{figure}[h]
\includegraphics[scale=0.2]{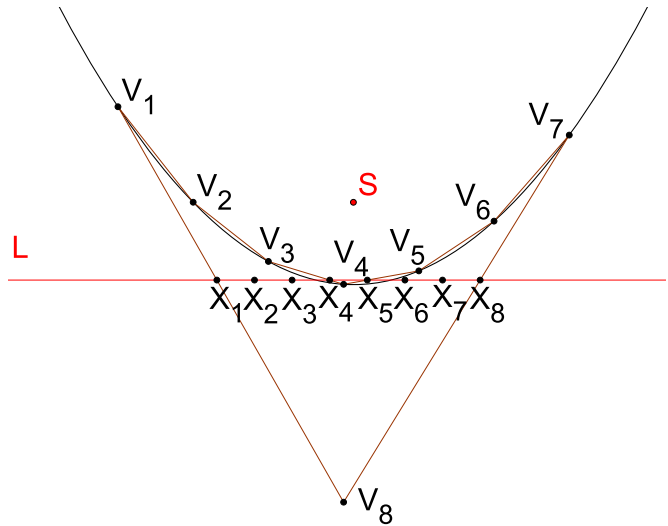}

\caption{\label{fig:equidistant Simson on parabola}Points $V_{1},...,V_{8}$
are the vertices of an equidistant Simson octagon with a Simson point
$S$, Simson line $L$ and projections $X_{1},...,X_{8}$. By Theorem
\ref{thm:Parabola}, $V_{1},...,V_{7}$ lie on a parabola.}
\end{figure}

\begin{cor}
\label{cor:Parabola Proj property}Let $C$ be a parabola with focus
$F$. The locus of projections of $F$ into the lines tangent to $C$
is the tangent to $C$ at its vertex.\end{cor}
\begin{proof}
As seen in the proof of Theorem \ref{thm:Parabola}, the coordinates
of the $V_{i}$, $i=1,...,n$ are continuous functions of $\Delta$.
Therefore as $n\rightarrow\infty$ and $\Delta\rightarrow0$ in Theorem
\ref{thm:Parabola}, the limit of the polygon is a parabola with focus
$S$ and tangent line at the vertex equal to $L$.
\end{proof}
This property can be equivalently stated as: ``the pedal curve of
the focus of a parabola with respect to the parabola is the line tangent
to it at its vertex''. This property is by no means new, but its
derivation does give a nice connection between the pedal of a polygon
and the pedal of the parabola. Specifically, we can view the focus
$F$ as the Simson point of a parabola (considered as a polygon with
infinitely many points) and the tangent at the vertex as the Simson
line of the parabola.

\begin{figure}[h]
\includegraphics[scale=0.45]{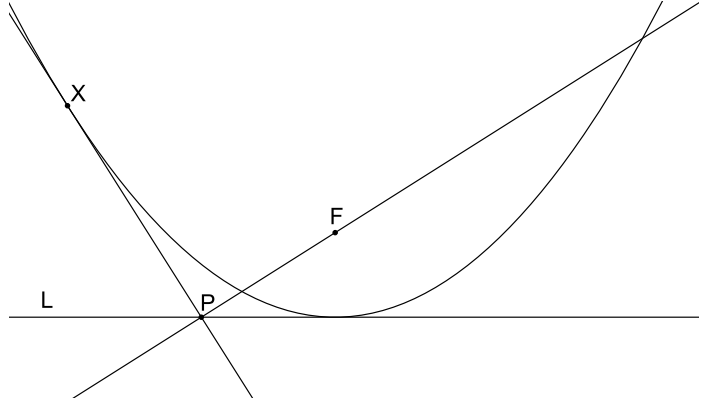}

\caption{Corollary \ref{cor:Parabola Proj property}: $X$ is a variable point
of $C$, $F$ is the focus, $P$ is the projection of $F$ into the
tangent at $X$ and $L$ is the tangent to $C$ at its vertex.}
\end{figure}

Let $V_{1}\cdot\cdot\cdot V_{n+2}$ be an equidistant Simson polygon.
We will now prove that the sides connecting the vertices $V_{1},V_{2},...,V_{n+1}$
form an optimal piecewise linear continuous approximation of the parabola.
To be precise, we show that it is a solution to the following problem:
\begin{problem}
Consider a continuous piecewise linear approximation $l(x)$ of a
parabola $f(x),$ $x\in[a,b]$ obtained by connecting several points
on the parabola. That is, let 
\[
l(x)=\frac{f(x_{i+1})-f(x_{i})}{x_{i+1}-x_{i}}(x-x_{i+1})+f(x_{i+1})\mbox{ for }x\in[x_{i},x_{i+1}]
\]

\noindent \begin{flushleft}
where $a=x_{0}<x_{1}<...<x_{n-1}<x_{n}=b$. Find $x_{1},x_{2},...,x_{n-1}\in(a,b)$
such that the error
\par\end{flushleft}

\[
\int_{a}^{b}|f(x)-l(x)|dx
\]

\noindent \begin{flushleft}
is minimal. 
\par\end{flushleft}
\end{problem}
The points $(x_{i},f(x_{i})),\mbox{ }i=0,...,n$ are called knot points
and a continuous piecewise linear approximation which solves the problem
is called optimal. Since all parabolas are similar, it suffices to
consider $f(x)=\frac{x^{2}-\Delta^{2}}{4s}$.
\begin{thm}
\label{thm:The-optimal-piecewise-continuous}The optimal piecewise-continuous
linear approximation to $f(x)$ with the setup above is given by the
sides $V_{1}V_{2}$,$V_{2}V_{3}$,...,$V_{n}V_{n+1}$ of an equidistant
Simson $(n+2)$-gon with $X_{1}=\frac{a}{2}$, $\Delta=\frac{b-a}{n}$
and $V_{i}=(a+(i-1)\Delta,f(a+(i-1)\Delta))$. The knot points $(x_{0},f(x_{0})),...,(x_{n},f(x_{n}))$
are the vertices $V_{1},V_{2},...,V_{n+1}$.\end{thm}
\begin{proof}
The equation of the $i$th line segment simplifies to 
\[
l(x)=\frac{x(x_{i+1}+x_{i})-x_{i}x_{i+1}-\Delta^{2}}{4s},\mbox{ for }x\in[x_{i},x_{i+1}].
\]

\noindent \begin{flushleft}
Therefore $f(x)-l(x)=\frac{(x-x_{i+1})(x-x_{i})}{4s}$ for $x\in[x_{i},x_{i+1}]$.
Integrating $|f(x)-l(x)|$ from $x_{i}$ to $x_{i+1}$ we get
\[
{\displaystyle \int_{x_{i}}^{x_{i+1}}|f(x)-l_{i}(x)|dx=\frac{(x_{i+1}-x_{i})^{3}}{24|s|}.}
\]

\par\end{flushleft}

\noindent \begin{flushleft}
It is enough to minimize
\[
S(x_{1},...,x_{n-1})=24|s|{\displaystyle \int}_{a}^{b}|f(x)-l(x)|dx=\sum_{i=0}^{n-1}(x_{i+1}-x_{i})^{3}.
\]

\par\end{flushleft}

\noindent \begin{flushleft}
Taking the partial derivative with respect to $x_{i}$ for $1\leq i\leq n-1$
and setting to zero, we get
\[
\frac{\partial}{\partial x_{i}}S(x_{1},...,x_{n-1})=3(x_{i}-x_{i-1})^{2}-3(x_{i+1}-x_{i})^{2}=0\iff(x_{i}-x_{i-1})^{2}=(x_{i+1}-x_{i})^{2}.
\]

\par\end{flushleft}

\noindent \begin{flushleft}
Since the points are ordered and distinct, $x_{i}=\frac{x_{i+1}+x_{i-1}}{2}$,
so that the $x_{i}$'s form an arithmetic progression. The $x$-coordinates
of the vertices $V_{i}$ satisfy this relation, and by uniqueness,
the theorem is proved. 
\par\end{flushleft}
\end{proof}
By similar reasoning, one can see that the same sides of the $(n+2)$-gon
are also optimal if the problem is modified to solving the least-squares
problem 
\[
\min_{x_{1},...,x_{n-1}}\int_{a}^{b}(f(x)-l(x))^{2}dx.
\]

$ $

From the proof of Theorem \ref{thm:The-optimal-piecewise-continuous},
we have the following interesting result about parabolas. 
\begin{cor}
Let $f(x)$ be the equation of parabola, $\Delta$ be a real number
and let $l(x)$ be the line segment with end points $(y,f(y)),(y+\Delta,f(y+\Delta))$.
Then the area

\[
\int_{y}^{y+\Delta}|f(x)-l(x)|dx
\]
bounded by $f(x)$ and $l(x)$ is independent of $y$.
\end{cor}
This property also explains why the $x$-coordinates of the knot points
of the optimal piecewise linear continuous approximation of the parabola
are at equal intervals. 

We now list some of the properties of equidistant Simson polygons:
\begin{thm}
\label{thm:Parallel chords}An equidistant Simson polygon $V_{1}V_{2}...V_{n}$
with projections $X_{1},X_{2},...,X_{n}$ has the following properties:\end{thm}
\begin{enumerate}
\item \label{enu:odd parallel segments}If $j-i>0$ is odd, the segments
$V_{i}V_{j}$, $V_{i+1}V_{j-1}$,...,$V_{\frac{j+i+1}{2}}V_{\frac{j+i-1}{2}}$
are parallel for every $i,j\in\{1,2,...,n-1\}$.
\item \label{enu:even parallel segments}If $j-i>0$ is even, the segments
$V_{i}V_{j}$, $V_{i+1}V_{j-1}$,...,$V_{\frac{j+i}{2}-1}V_{\frac{j+i}{2}+1}$
and the tangent to the parabola at $V_{\frac{j+i}{2}}$ are parallel
for every $i,j\in\{1,2,...,n-1\}$.
\item \label{enu:midpoints collinear}The midpoints of the parallel segments
in (\ref{enu:odd parallel segments}) (respectively (\ref{enu:even parallel segments}))
lie on a line orthogonal to the Simson line $L$.\end{enumerate}
\begin{proof}
(\ref{enu:odd parallel segments}). The slope between $V_{i}$ and
$V_{j}$ is easily calculated to be $\frac{2X+(i+j-1)\Delta}{2s}$.

(\ref{enu:even parallel segments}). Recall that the parabola is given
by $y=\frac{x^{2}-\Delta^{2}}{4s}$ so that its slope at $V_{\frac{j+i}{2}}$
is $\frac{2X+(2(\frac{j+i}{2})-1)\Delta}{2}=\frac{2X+(j+i-1)\Delta}{2}$.

(\ref{enu:midpoints collinear}). The $x$-coordinate of the midpoint
of $V_{i}V_{j}$ is $2X+(i+j-1)\Delta$.
\end{proof}
\begin{figure}[h]
\includegraphics[scale=0.2]{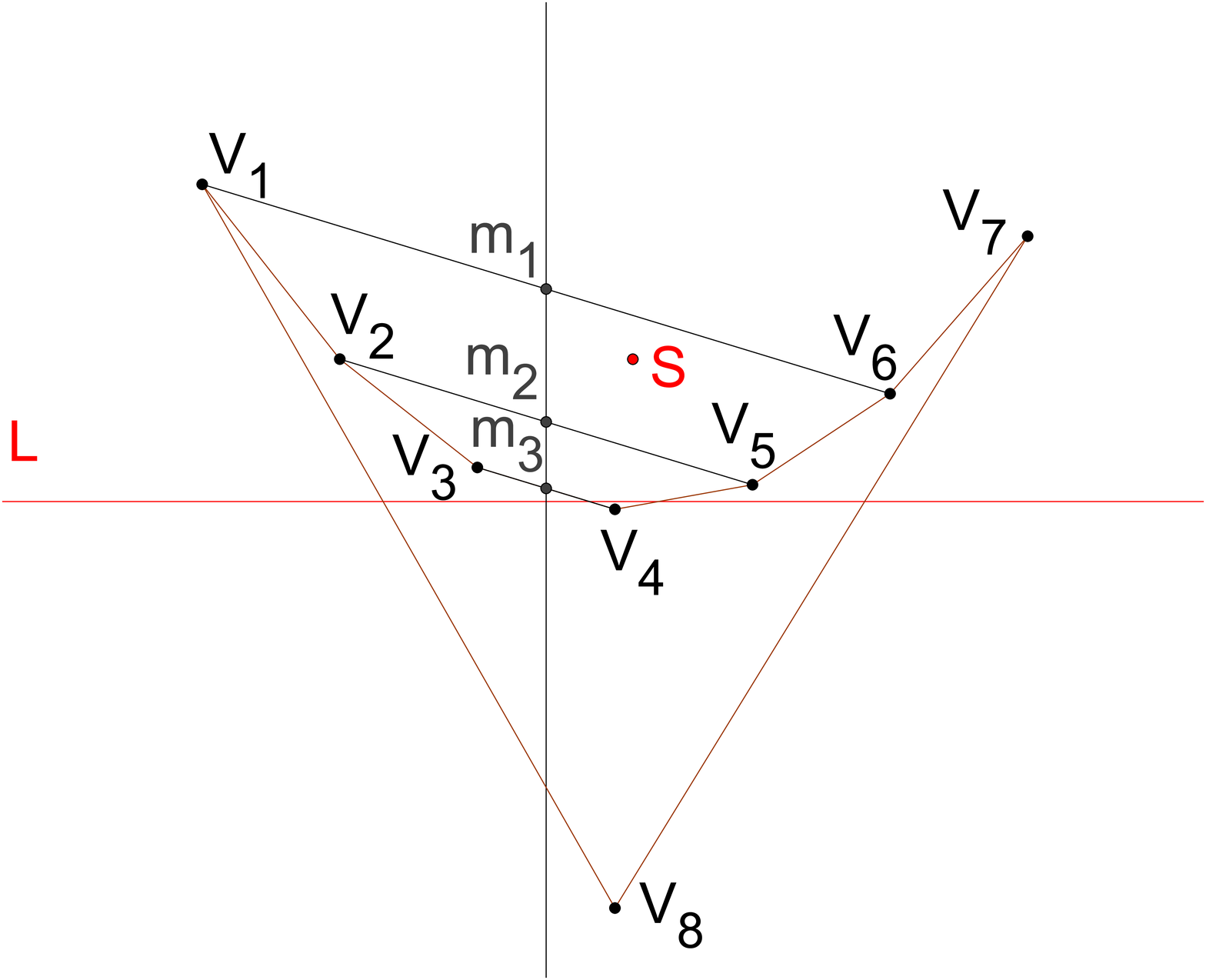}

\caption{Points $V_{1},...,V_{8}$ are the vertices of an equidistant Simson
octagon with Simson point $S$ and Simson line $L$. By Theorem \ref{thm:Parallel chords},
the segments $V_{1}V_{6}$, $V_{2}V_{5}$ and $V_{3}V_{4}$ are parallel,
and their midpoints $m_{1}$, $m_{2}$ and $m_{3}$ all lie on a line
perpendicular to $L$. }
\end{figure}

The following property of Simson polygons can be viewed as a discrete
analog of the isogonal property of the parabola.
\begin{prop}
\label{prop:discrete reflection property}Let $S$ and $L$ be the
Simson point and Simon line of a Simson polygon (not necessarily equidistant)
with vertices $V_{1},...,V_{n}$ and define $X_{1},...,X_{n}$ as
before. Let $V_{i}'$ be the reflection of $V_{i}$ in $L$. Then
the lines $V_{i}X_{i}$ and $V_{i}X_{i+1}$ are isogonal with respect
to the lines $V_{i}V_{i}'$ and $V_{i}S$ (i.e. $\angle V_{i}'V_{i}X_{i}=\angle X_{i+1}V_{i}S$)
for $i=1,...,n$.\end{prop}
\begin{proof}
The proof is by a straightforward angle count.
\end{proof}
In the case when the Simson polygon in Proposition \ref{prop:discrete reflection property}
is equidistant, we can take limits to obtain the isogonal property
of the parabola:
\begin{cor}
\label{cor:focal properpty reflect}Let $C$ be a parabola with focus
$F$ and tangent line $L$ at its vertex. Let $X$ be any point on
$C$ and $K$ the tangent at $X$. Furthermore, let $X'$ be the reflection
of $X$ in $L$. Then $K$ forms equal angles with $X'X$ and $FX$.
\end{cor}
\begin{figure}[h]
\includegraphics[scale=0.3]{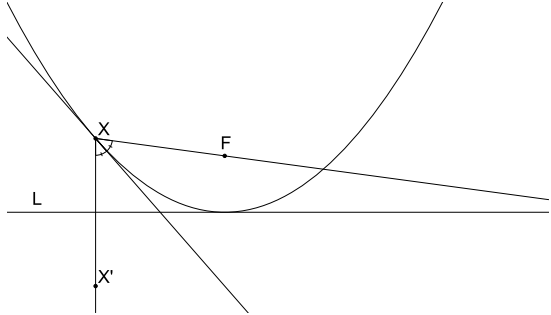}

\caption{Corollary \ref{cor:focal properpty reflect}: $X$ is a variable point
of the parabola $C$, $F$ is the focus, $L$ is the tangent to $C$
at its vertex and $X'$ is the reflection of $X$ in $L$. The lines
$XF$ and $XX'$ form equal angles with the tangent at $X$.}
\end{figure}

Using the same setup as in Theorem \ref{thm:Parabola} for an equidistant
Simson polygon,
\begin{thm}
\label{thm:midpoint parabola}Let $M_{i}$ be the midpoint of $V_{i}V_{i+1}$,
$i=1,...,n-2$. Then the midpoints $M_{i}$ lie on a parabola $C'$
with focus $S$ and tangent line at its vertex $L$.\end{thm}
\begin{proof}
Since $V_{i}=(2X+(2i-1)\Delta,\frac{(x+(i-1)\Delta)(x+i\Delta)}{s})$,
\[
M_{i}=(2(X+i\Delta),\frac{(X+i\Delta)^{2}}{s}).
\]

\noindent \begin{flushleft}
Therefore the $M_{i}$ lie on the parabola $p(x)=\frac{x^{2}}{4s}$
with focus $S$. The slope of $V_{i}V_{i+1}$ is $\frac{X+i\Delta}{s}$,
which is the same as that of $p(x)$ at $M_{i}$.
\par\end{flushleft}
\end{proof}
In a coordinate system where $S$ lies above $L$, the parabolas $C$
and $C'$ form sharp upper and lower bounds to the piecewise linear
curve $f(x)$ formed by the sides connecting $V_{1},...,V_{n-1}$
(discussed in Theorem \ref{thm:The-optimal-piecewise-continuous}).
Informally, one can think of $C$ and $C'$ as ``sandwiching'' $f(x)$,
and in the limit $n\rightarrow\infty$ and $\Delta\rightarrow0$,
the two curves coincide and equal the limit of the polygon.

The following result is a discrete analog of the famous optical reflection
property of the parabola. 
\begin{cor}
\label{cor:main focal}Let $M_{i}$ be the midpoints of $V_{i}V_{i+1}$
as in Theorem \ref{thm:midpoint parabola} and $p_{i}$ be the line
passing through $M_{i}$ orthogonal to $L$ for $i=1,2,...,n-2$.
Then the reflection $p_{i}'$ of $p_{i}$ in $V_{i}V_{i+1}$ passes
through $S$ for each $i=1,2,...,n-2$.
\end{cor}
\begin{figure}[h]
\includegraphics[scale=0.2]{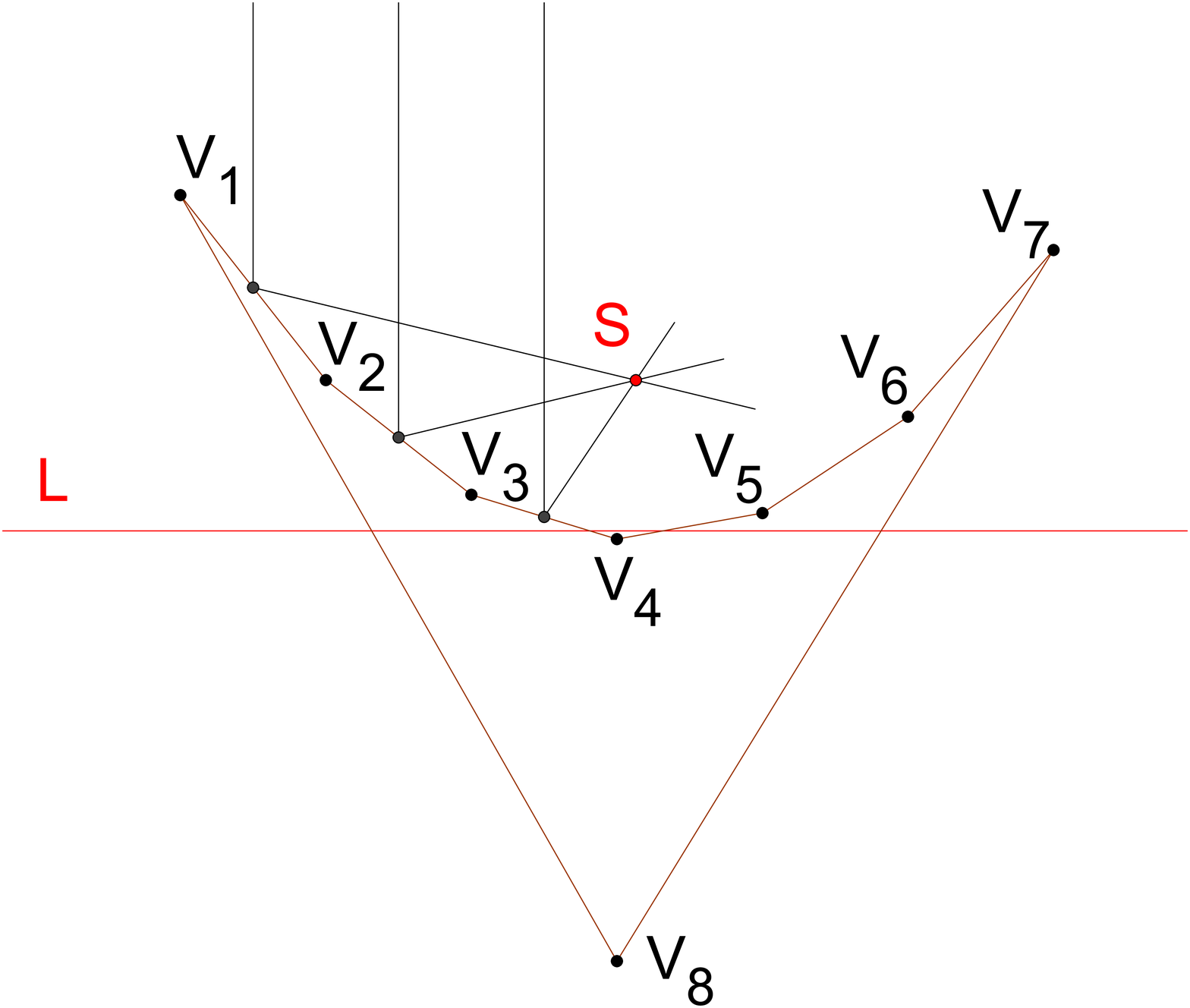}

\caption{Corollary \ref{cor:main focal}: $\Pi=V_{1}\cdot\cdot\cdot V_{8}$
is an equidistant Simson polygon. The reflections at the midpoints
of the sides of $\Pi$ of rays orthogonal to $L$ pass through $S$.}
\end{figure}

Let $X$ and $Y$ be two points on a parabola $C$. The triangle formed
by the two tangents at $X$ and $Y$ and the chord connecting $X$
and $Y$ is called an \textit{Archimedes Triangle \cite{Bogomolny Parabola}.
}The chord of the parabola is called the triangle's base. One of the
results stated in Archimedes' Lemma is that if $Z$ is the vertex
opposite to the base of an Archimedes triangle and $M$ is the midpoint
of the base, then the median $MZ$ is parallel to the axis of the
parabola. The following result yields a discrete analog to Archimedes'
Lemma. Let $V_{1}\cdot\cdot\cdot V_{n}$ be an equidistant Simson
polygon.
\begin{thm}
Let $W_{i,j}=V_{i}V_{i+1}\cap V_{j}V_{j+1}$ for each $i,j\in\{1,2,...,n-2\}$
and $i\neq j$. Let $M_{i,j+1}$ and $M_{i+1,j}$ be the respective
midpoints of chords $V_{i}V_{j+1}$ and $V_{i+1}V_{j}$. Then $W_{i,j}M_{i,j+1}$
and $W_{i,j}M_{i+1,j}$ are orthogonal to $L$.\end{thm}
\begin{proof}
As shown in the proof of Theorem \ref{thm:Parallel chords}, the $x$-coordinate
of $M_{i,j+1}$ is $2X+(i+j)\Delta$ and that of $M_{i+1,j}$ is the
same. The point $W_{i,j}$ is the intersection of the line $V_{i}V_{i+1}$
given by $\mbox{ }y=\frac{X+i\Delta}{s}x-\frac{(X+i\Delta)^{2}}{s}$
and the line $V_{j}V_{j+1}$ given by $\mbox{ }y=\frac{X+j\Delta}{s}x-\frac{(X+j\Delta)^{2}}{s}$,
so that $W_{i,j}=(2X+(i+j)\Delta,\frac{(X+i\Delta)(X+j\Delta)}{s})$.\end{proof}
\begin{cor}
The points $W_{i,j+1},W_{i+1,j}$, $W_{i+2,j-1}$, etc. and the points
$M_{i,j+1}$, $M_{i+1,j}$, $M_{i+2,j-1}$, are collinear. The line
on which they lie is orthogonal to $L$.
\end{cor}
Taking limits, we get the following Corollary which includes the part
of Archimedes\textquoteright{} Lemma stated previously:
\begin{cor}
\label{cor:Archimedes generalized}The vertices opposite to the bases
of all Archimedes triangles with parallel bases lie on a single line
parallel to the axis of the parabola and passing through the midpoints
of the bases.
\end{cor}
\begin{figure}[h]
\includegraphics[scale=0.4]{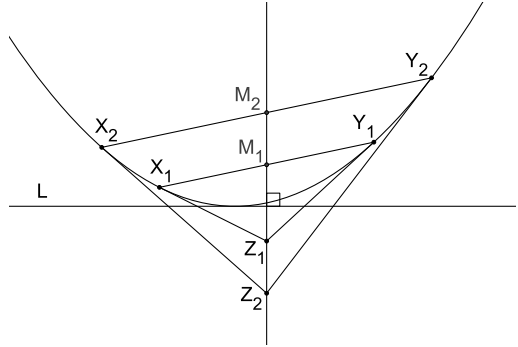}

\caption{Corollary \ref{cor:Archimedes generalized}: $\triangle X_{1}Y_{1}Z_{1}$
and $\triangle X_{2}Y_{2}Z_{2}$ are two Archimedes triangles with
parallel bases $X_{1}Y_{1}$, $X_{2}Y_{2}$. Points $Z_{1}$,$Z_{2}$
and the midpoints of the bases $M_{1},M_{2}$ all lie on a line parallel
to the axis of the parabola.}
\end{figure}

The final theorem to which we give generalization is \textit{Lambert's
Theorem}, which states that the circumcircle of a triangle formed
by three tangents to a parabola passes through the focus of the parabola
\cite{Bogomolny Parabola}. We can prove it using the Simson-Wallace
Theorem.

\begin{figure}[h]
\includegraphics[scale=0.35]{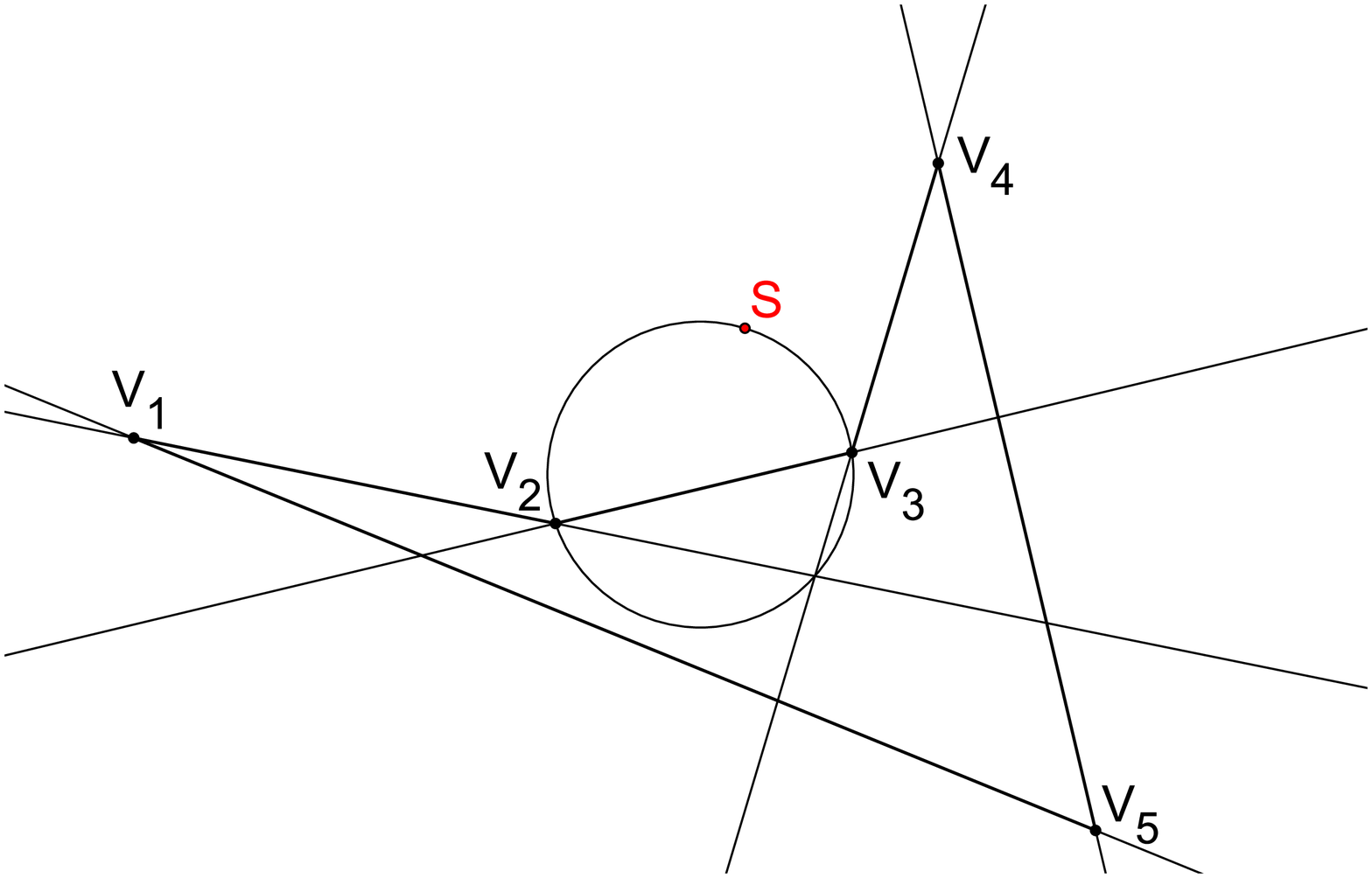}\includegraphics[scale=0.35]{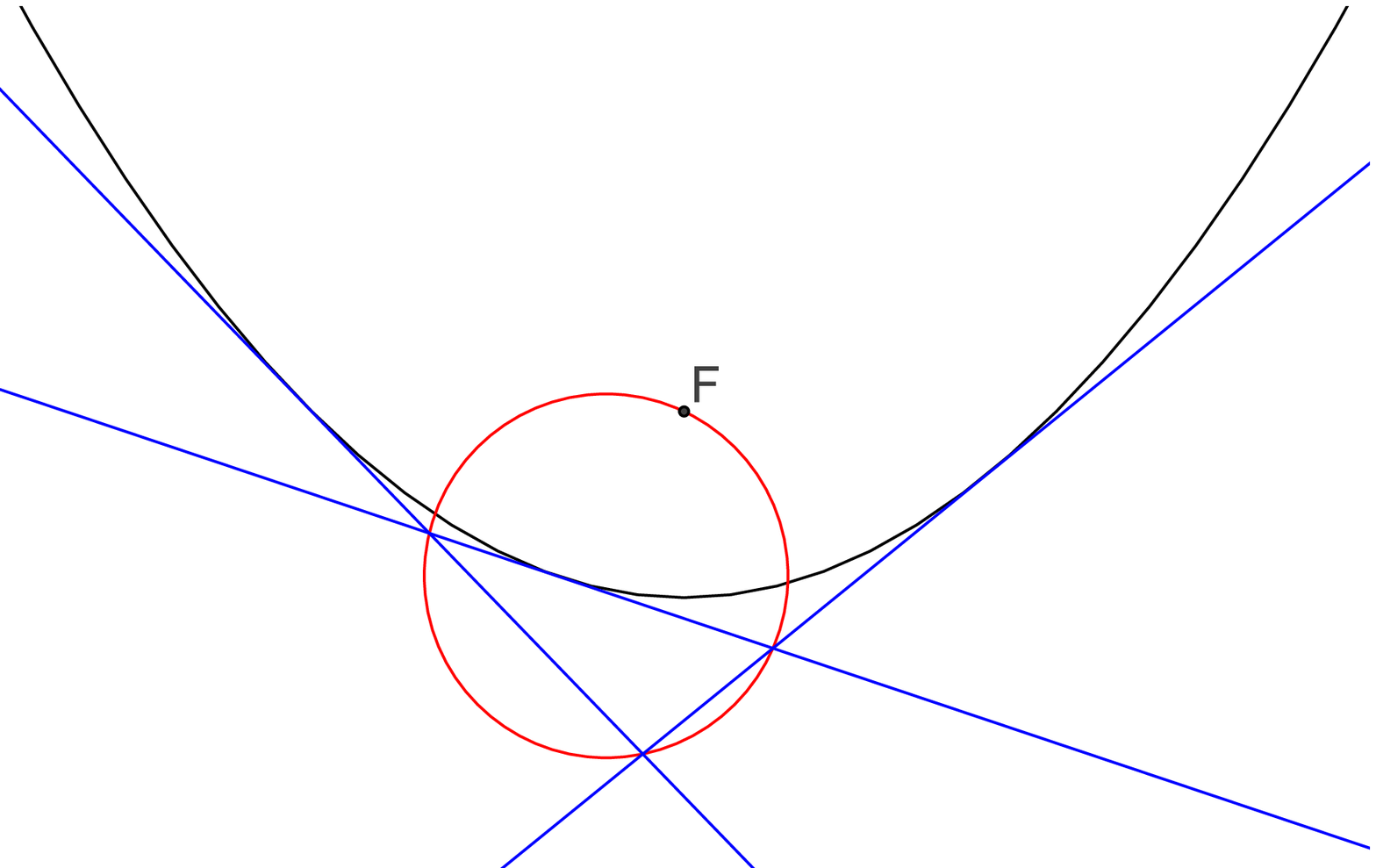}

\caption{Theorem \ref{thm:Discrete lambert} and Corollary \ref{cor:Lambert's-Theorem.}.}
\end{figure}

\begin{thm}
\label{thm:Discrete lambert}Let $V_{1}\cdot\cdot\cdot V_{n}$ be
a Simson polygon (not necessarily equidistant) with Simson point $S$.
Let $i,j,k\in\{1,2,...,n\}$, be distinct. Then the circumcircle of
the triangle $T$ formed from lines $V_{i}V_{i+1}$, $V_{j}V_{j+1}$
and $V_{k}V_{k+1}$ passes through $S$.\end{thm}
\begin{proof}
Since the projections of $S$ into $V_{i}V_{i+1}$, $V_{j}V_{j+1}$
and $V_{k}V_{k+1}$ are collinear, $S$ is a Simson point of the triangle
$T$. Therefore by the Simson-Wallace Theorem (Theorem \ref{(Simson-Wallace-Theorem)}),
$S$ lies on the circumcircle of $T$.\end{proof}
\begin{cor}
\label{cor:Lambert's-Theorem.}(Lambert's Theorem). The focus of a
parabola lies on the circumcircle of a triangle formed by any three
tangents to the parabola.\end{cor}
\begin{proof}
Taking the limit of a sequence of equidistant Simson polygons gives
Lambert's Theorem for a parabola, since the lines $V_{i}V_{i+1}$,
$V_{j}V_{j+1}$, $V_{k}V_{k+1}$ become tangents in the limit.
\end{proof}

\section{Acknowledgments}

The author would like to express his gratitude to Olga Radko for valuable
feedback throughout the process of writing this paper. 

$ $

Emmanuel Tsukerman: Stanford University

\textit{E-mail address: emantsuk@stanford.edu}
\end{document}